\documentclass[12pt]{amsart}

 \usepackage{amsfonts,graphics,amsmath,amsthm,amsfonts,amscd,amssymb,amsmath,latexsym,multicol,
 mathrsfs}
\usepackage{epsfig,url}
\usepackage{flafter}
\usepackage{fancyhdr}
\usepackage{hyperref}
\hypersetup{colorlinks=true, linkcolor=black}

\makeatletter

\def\jobis#1{FF\fi
  \def\predicate{#1}%
  \edef\predicate{\expandafter\strip@prefix\meaning\predicate}%
  \edef\job{\jobname}%
  \ifx\job\predicate
}

\makeatother

\if\jobis{proposal}%
\else
\fi

 \usepackage[matrix, arrow]{xy}
 
\DeclareMathOperator{\Bs}{Bs}

\DeclareMathOperator{\Div}{Div}

 \numberwithin{equation}{subsection}
 \numberwithin{footnote}{subsection}

 \newtheorem{lem}[subsection]{Lemma}
 \newtheorem{prop}[subsection]{Proposition}
 \newtheorem{thm}[subsection]{Theorem}

 \newtheorem{quest}[subsection]{Question}
 
{
    \newtheoremstyle{upright}%
        {8pt plus2pt minus4pt}%
        {8pt plus2pt minus4pt}%
        {\upshape}%
        {}%
        {\bfseries\scshape}%
        {}%
        {1em}%
        {}%
\theoremstyle{upright}

 \newtheorem{defn}[subsection]{Definition}
 
 \newtheorem{exa}[subsection]{Example}

}

 \newcommand{\C}{\mathbb C}
 \newcommand{\N}{\mathbb N}
 \newcommand{\PP}{\mathbb P}
 
 \newcommand{\Q}{\mathbb Q}
 \newcommand{\R}{\mathbb R}
 \newcommand{\Z}{\mathbb Z}
 
 \newcommand{\rddown}[1]{\left\lfloor{#1}\right\rfloor} 
 \newcommand{\nrd}[1]{\langle{#1}\rangle}
 
\title{\large T\MakeLowercase{he augmented base locus of real divisors over arbitrary fields}}
\author{\large C\MakeLowercase{aucher} B\MakeLowercase{irkar}}
\date{\today}
\begin{document}
\maketitle

\begin{abstract} 
We show that the augmented base locus coincides with the exceptional 
locus (i.e. null locus) for any nef $\R$-Cartier divisor on any scheme projective  
over a field (of any characteristic). 
Next we prove a semi-ampleness criterion in terms of the augmented base locus 
generalizing a result of Keel. We also study nef divisors with positive top intersection number, and   
discuss some problems related to augmented base loci of log divisors.
\end{abstract}



\section{Introduction}

The base locus of a linear system is a fundamental notion in algebraic and especially 
birational geometry. The restricted base locus (also called the non-nef locus) and the augmented base locus 
(also called the non-ample locus) are 
refinements of the base locus which capture more essential properties of divisors and linear systems. These are 
closely related to important concepts and problems in birational geometry, eg see 
[\ref{BDPP}][\ref{Nakamaye}][\ref{Nakamaye-2}][\ref{ELMNP}][\ref{ELMNP-2}][\ref{CMM}][\ref{BH}]. 

We start with some definitions.\\

{\textbf{The augmented base locus.}}
Let $X$ be a scheme. An \emph{$\R$-Cartier divisor}
is an element of $\Div(X)\otimes_\Z\R$ where $\Div(X)$ is the group of Cartier divisors. 
A \emph{$\Q$-Cartier divisor} is defined similarly by tensoring with $\Q$.

\begin{defn}\label{d-abl}
Let $X$ be a projective scheme over a field $k$.
The \emph{stable base locus} of a $\Q$-Cartier divisor $L$ is defined as 
$$
{\bf{B}}(L)=\bigcap_{m\in \N, ~mL ~\mbox{Cartier}} \Bs|mL|
$$ 
that is, it is the set of 
those points $x\in X$ such that every section of every $mL$ vanishes where $m$ is 
a positive integer and $mL$ is Cartier. The base locus, stable base locus, and all the other base loci defined 
below are considered with the reduced induced structure.
The \emph{augmented base locus} of  $L$ is defined as 
$$
{\bf{B_+}}(L)=\bigcap_{m\in \N} {\bf{B}}({mL}-A)
$$
where $A$ is any  ample Cartier divisor. 
\end{defn}

The augmented base locus of $\R$-Cartier divisors on smooth projective varieties was 
defined in [\ref{ELMNP-2}]. 
For basic properties of the augmented base locus in this context see [\ref{ELMNP}][\ref{ELMNP-2}].
We give a different definition which is more convenient 
for our purposes (the two definitions agree, by Lemma \ref{l-independence} (3) below).

\begin{defn}\label{d-abl-2}
Let $X$ be a projective scheme over a field $k$.
Let $L$ be an $\R$-Cartier divisor on $X$. We can write $L\sim_\R\sum t_iA_i$ 
where $A_1,\dots, A_r$ are very ample Cartier divisors and $t_i\in \R$. The $A_i$ are not necessarily distinct 
and the expression is obviously not unique. 
Define $\langle mL \rangle=\sum \rddown{mt_i}A_i$ which depends on the above expression.  
Next define the augmented base locus of $L$ as 
$$
{\bf{B_+}}(L)=\bigcap_{m\in \N} {\bf{B}}(\nrd{mL}-A)
$$
where $A$ is any ample Cartier divisor. 
\end{defn}

It turns out that ${\bf{B_+}}(L)$ does not depend on the 
choice of $A$ nor the $A_i$ nor the expression $L\sim_\R\sum t_iA_i$ (see Lemma \ref{l-independence}).\\ 

{\textbf{Relation with the exceptional locus.}}
Let $X$ be a projective scheme over a field $k$ and $L$ an $\R$-Cartier divisor on $X$. 
The \emph{exceptional locus} of $L$ (also called the null locus when $L$ is nef) is defined as 
$$
\mathbb{E}(L):=\bigcup_{\mbox{$L|_V$ not big}} V
$$
where the union runs over the integral subschemes $V\subseteq X$ with positive dimension.

Now we come to the first result of this paper.

\begin{thm}\label{t-main-1}
Let $X$ be a projective scheme over a field $k$. Assume that $L$ is a nef  $\R$-Cartier divisor 
with a given expression $L\sim_\R\sum t_iA_i$ as in \ref{d-abl-2}, 
and $A$ is a very ample Cartier divisor on $X$. Then
$$
{\bf{B_+}}(L)={\bf{B}}(\nrd{nL}-A)={{\Bs}}|\nrd{nL}-A|=\mathbb{E}(L)
$$
for any sufficiently divisible $n\in \N$. 
\end{thm}

The theorem was first proved for $X$ smooth, $\Q$-Cartier $L$, and $k$ algebraically closed of characteristic zero 
by Nakamaye [\ref{Nakamaye}] using Kodaira 
type vanishing theorems, and this was generalized to $\R$-Cartier divisors by 
Ein-Lazarsfeld-Musta\c{t}\u{a}-Nakamaye-Popa [\ref{ELMNP-2}].
Nakamaye's result was extended to log canonical varieties by 
Cacciola-Felice Lopez [\ref{CF}] again by using Kodaira type vanishing theorems. They also give some 
applications to the moduli spaces of curves. 
Related results concerning the restricted volume are 
proved on normal varieties by 
Boucksom-Cacciola-Felice Lopez [\ref{BCF}].
 
The theorem was proved by Cascini-M$^{\rm c}$Kernan-Musta\c{t}\u{a} [\ref{CMM}] when $k$ 
is algebraically closed of positive characteristic 
using techniques related to Keel [\ref{Keel}]: 
the main ingredients are Serre vanishing and the Frobenius. Fujino-Tanaka [\ref{FT}] employ similar 
arguments on surfaces using Fujita vanishing and the Frobenius. 
We will also use Fujita vanishing (but not the Frobenius).\\

{\textbf{A semi-ampleness criterion.}}
The following semi-ampleness result was first proved by Keel [\ref{Keel}] when $k$ has positive characteristic.
A simplified proof of Keel's result was given by Cascini-M$^{\rm c}$Kernan-Musta\c{t}\u{a} [\ref{CMM}].

\begin{thm}\label{t-main-2}
Let $X$ be a projective scheme over a field $k$. Assume that $L$ is a nef $\Q$-Cartier divisor on $X$. 
Then there is a closed subscheme $Z\subseteq X$ such that 

$\bullet$ the reduced induced scheme associated to $Z$ is equal to $\mathbb{E}(L)$, and  

$\bullet$ $L$ is semi-ample if and only if $L|_{Z}$ is semi-ample.
\end{thm}

 When $k$ has positive characteristic we can use the Frobenius to show that 
in fact we can take $Z=\mathbb{E}(L)$. However, when $k$ has characteristic zero  
in general we cannot take $Z=\mathbb{E}(L)$, by Keel [\ref{Keel}, \S 3]. Although $Z$ is not 
unique but some choice can be calculated for any given $X,L$. It is interesting to see whether 
the theorem holds if $L$ is only $\R$-Cartier.\\

{\textbf{Divisors with positive top intersection number.}}
Assume that $X$ is a normal projective variety of dimension $d$ over an algebraically closed field $k$, 
and $L$ a nef $\R$-Cartier divisor with $L^d>0$. 
If $L$ is $\Q$-Cartier, then it is a well-known fact that $L$ is big. 
When $L$ is only $\R$-Cartier, Shokurov [\ref{Sh-log-models}, Lemma 6.17] shows that 
$L$ is big if $k$ has characteristic zero, using resolution of singularities and Kawamata-Viehweg vanishing 
(although he only proves that $L\sim_\Q M\ge 0$ but his proof can be extended to show that 
$L$ is big). Here 
big means that $L\sim_\R A+D$ where $A$ is an ample $\R$-Cartier divisor and $D$ is 
an effective $\R$-Cartier divisor.   
In [\ref{Keel}, Remark 5.5.3], Keel asks whether the same holds in positive characteristic. We show that 
indeed it holds in a quite general setting. 

\begin{thm}\label{t-main-3}
Let $X$ be a projective variety (i.e. projective integral scheme) of dimension $d$ over a field $k$. 
Assume that $L$ is a nef $\R$-Cartier divisor on $X$. 
Then $L$ is big if and only if $L^d>0$.
\end{thm}

As Keel points out this result has interesting applications to the structure of extremal rays 
on varieties. In fact the characteristic zero case already plays a crucial role in the 
study of extremal rays on the moduli space of 
pointed rational curves  (see Keel-M$^{\rm c}$Kernan [\ref{KMc}, Lemma 2.2] and its many uses therein).\\

{\textbf{The augmented base locus of log divisors.}}
Let $(X,B)$ be a projective pair over an algebraically closed field $k$ and $A$ a nef 
and big $\R$-divisor such that $L=K_X+B+A$ is nef. The locus ${\bf{B}_+}(L)$ is 
closely related to the geometry of $X$. In Section 6 we recall some results and pose some 
questions concerning such loci.\\

{\textbf{Acknowledgements.}} 
I would like to thank Mircea Musta\c{t}\u{a} and Karl Schwede for discussions related to Section 6. 
This work was supported by a Leverhulme grant.\\

\section{Preliminaries}

\subsection{The operator $\nrd{-}$}\label{ss-<>} 
Let $L$ be an $\R$-Cartier divisor with an expression $L\sim_\R \sum t_iA_i$ as in 
\ref{d-abl-2}. Let $\pi\colon X'\to X$ be a morphism such that $A_i':=\pi^*A_i$ is 
very ample for each $i$, eg $\pi$ is a closed embedding or it is obtained by 
base change as in \ref{ss-b-change} below.
Then we get the expression $L':=\pi^*L\sim_\R \sum t_iA_i'$ which we can use to 
define $\nrd{mL'}$. It is clear that $\nrd{mL'}=\pi^*\nrd{mL}$.

For a coherent sheaf $\mathcal{F}$ on $X$, we often use the notation 
$\mathcal{F}\nrd{mL}$ instead of $\mathcal{F}(\nrd{mL})$.

\subsection{{Growth of functions.}}
Let $h\colon \Z\to \Z$ be a function. We say that the \emph{upper growth} of $h$ is like $m^d$ 
(resp. at most like $m^d$) if 
$$
0<\limsup_{m\to +\infty} \frac{h(m)}{m^d}<+\infty
$$ 
(resp. $\limsup_{m\to +\infty} \frac{h(m)}{m^d}<+\infty$).

\subsection{{Divisors.}}
Let $X$ be a projective scheme over a field $k$. The group of Cartier divisors on $X$ 
 is denoted by $\Div(X)$ (this group modulo linear equivalence is denoted 
${\rm{Cl}}(X)$). Recall that an $\R$-Cartier divisor 
(resp. $\Q$-Cartier divisor) is an element of $\Div(X)\otimes_\Z \R$ (resp. $\Div(X)\otimes_\Z \Q$).
Such a divisor can be represented as $L=\sum l_iL_i$ where $l_i\in \R$ (resp. $l_i\in \Q$) 
and $L_i$ are Cartier divisors but this representation is not unique. Two $\R$-Cartier divisors $L,L'$ 
are \emph{$\R$-linearly equivalent} (resp. \emph{$\Q$-linearly equivalent}) if $L-L'=\sum a_iN_i$ 
where $a_i\in \R$ (resp. $a_i\in \Q$)
and $N_i$ are Cartier divisors linearly equivalent to zero. We denote the equivalence by 
$L\sim_\R L'$ (resp. $L\sim_\Q L'$). 

An $\R$-Cartier divisor $L$ is: 

$\bullet$ \emph{nef} if $L\cdot C\ge 0$ for every 
curve $C\subseteq X$ (a curve is an integral closed subscheme of dimension one);

$\bullet$ \emph{ample} if 
$L\sim_\R\sum l_iL_i$ with $l_i> 0$ and $L_i$ ample Cartier divisors;

$\bullet$ \emph{effective} if 
$L=\sum l_iL_i$ with $l_i\ge 0$ and $L_i$ effective Cartier divisors;

$\bullet$ \emph{big} if $L\sim_\R A+D$ where $A$ is an ample 
$\R$-Cartier divisor and $D$ is an effective $\R$-Cartier divisor;

$\bullet$ \emph{semi-ample} if 
there is a projective morphism $f\colon X\to Y$ over $k$ and an ample 
$\R$-Cartier divisor $H$ on $Y$ such that $L\sim_\R f^*H$.

\subsection{{Pairs.}} 
A pair $(X,B)$ over a field $k$ consists of a normal quasi-projective 
variety over $k$ and a Weil $\R$-divisor $B$ with coefficients in $[0,1]$ such that 
$K_X+B$ is $\R$-Cartier where $K_X$ is the canonical divisor. The pair is 
\emph{klt} if for every projective birational morphism $f\colon Y\to X$ from a normal 
variety the coefficients of $B_Y$ are all $<1$ where 
$K_Y+B_Y=f^*(K_X+B)$.

\subsection{{Fujita vanishing theorem.}}\label{ss-Fujita-v}
This is a generalization of Serre vanishing theorem.
Let $X$ be a projective scheme over a field $k$, $A$ an ample Cartier divisor, and 
$\mathcal{F}$ a coherent sheaf on $X$. Then there is a number $m_0$ such that 
$h^i(\mathcal{F}(mA+L))=0$ for any $i>0$, $m>m_0$, and nef Cartier divisor $L$ 
[\ref{Fujita}][\ref{Laz}, Theorem 1.4.35].

\subsection{{Restriction to a hyperplane section.}}\label{ss-restriction}
Let $X$ be a projective scheme over a field $k$, $A$ an effective Cartier divisor, and 
$\mathcal{F}$ a coherent sheaf on $X$. Tensoring $\mathcal{F}$ with the exact sequence 
$$
0\to \mathcal{O}_X(-A)\to \mathcal{O}_X \to \mathcal{O}_A \to 0
$$ 
gives a sequence 
$$
0\to \mathcal{F}(-A)\to \mathcal{F} \to \mathcal{F}\otimes \mathcal{O}_A\to 0
$$
which is often not exact on the left. However, if $A$ (considered as a closed subscheme) 
does not contain any of the finitely many associated points of $\mathcal{F}$, then 
the latter sequence is also exact on the left. 

Now if $A$ is a very ample Cartier divisor and if $k$ is infinite, then perhaps after changing 
$A$ up to linear equivalence we can make sure that $A$ does not contain 
any associated point of $\mathcal{F}$. To prove this we first reduce it to the case $X=\PP^n_k$ 
and $\mathcal{O}_X(A)=\mathcal{O}_X(1)$. Then since $k$ is infinite, there are infinitely 
many hyperplanes defined over $k$ (hence infinitely many choices of $A$ up to linear equivalence) 
so we can avoid the associated points of $\mathcal{F}$. If $k$ is not infinite we will 
do a base change to an infinite field to be able to use the above remarks.

\subsection{Base loci and base change}\label{ss-b-change}
Let $X$ be a projective scheme over a field $k$ and let $L$ be a Cartier divisor on $X$. 
Recall that the \emph{base locus} of $L$ is defined as 
$$
\Bs|L|=\{x\in X \mid \mbox{$\alpha$ vanishes at $x$ for every $\alpha\in H^0(\mathcal{O}_X(L))$}\}
$$ 
As pointed out earlier we consider $\Bs|L|$ (and other loci) with the reduced structure. 
Recall that ${\bf{B}}(L)=\bigcap_{m\in \N} \Bs|mL|$. If $n,n'\in\N$, then each section 
$\alpha\in H^0(\mathcal{O}_X(nL))$ gives a section $\alpha^{\otimes n'}\in H^0(\mathcal{O}_X(n'nL))$
hence $\Bs|n'nL|\subseteq \Bs|nL|$. In particular, 
${\bf{B}}(L)=\Bs|mL|$ for every sufficiently divisible $m>0$.

Assume that $k\subseteq k'$ is a field extension and that $X'$ and $L'$ are the scheme and Cartier 
divisor obtained by base change to $k'$. Let $\pi\colon X'\to X$ be the corresponding morphism. 
Since 
$$
H^0(\mathcal{O}_X(L'))=H^0(\mathcal{O}_X(L))\otimes_kk'
$$ 
we can see that $\pi^{-1}\Bs|L|=\Bs|L'|$. 
This in turn implies that $\pi^{-1}{\bf{B}}(L)={\bf{B}}(L')$.

Now assume that $L$ is $\R$-Cartier with a given expression $L\sim_\R \sum t_iA_i$ 
as in \ref{d-abl-2}. As pointed out in \ref{ss-<>}, $\nrd{mL'}=\pi^*\nrd{mL}$ hence 
 $\pi^{-1}{\bf{B}}_+(L)={\bf{B}}_+(L')$.
 
With a little more work we can also see that $\pi^{-1}\mathbb{E}(L)\supseteq\mathbb{E}(L')$.
Indeed, let $V'$ be a component of $\mathbb{E}(L')$, let $W$ be the closure of $\pi(V')$, 
and let $W'$ be the scheme obtained from $W$ by base change. 
If $L|_W$ is big then $L|_W\sim_\R A_W+D_W$ where $A_W$ is ample and $D_W$ is 
effective. But then $L'|_{W'}\sim_\R A'_{W'}+D_{W'}'$ where $A_{W'}'$ is ample and 
$D_{W'}'$ is effective. Now $V'\nsubseteq D_{W'}'$ otherwise $W\subseteq D_W$ which is not 
possible. So by restricting to $V'$ we get $L'|_{V'}\sim_\R A'_{V'}+D_{V'}'$ which 
means that $L'|_{V'}$ is big, a contradiction. Therefore $\pi^{-1}\mathbb{E}(L)\supseteq \mathbb{E}(L')$.
So if in some situation we want to show that ${\bf{B}}_+(L)\subseteq \mathbb{E}(L)$, 
then it is enough to show that ${\bf{B}}_+(L')\subseteq \mathbb{E}(L')$ because $\pi$ is surjective.

\section{The augmented base locus is well-defined}

In this section, we show that the augmented base locus as defined in Definition \ref{d-abl-2} 
is well-defined. We also show that the definition agrees 
with \ref{d-abl} and the one in [\ref{ELMNP-2}].

\begin{lem}\label{l-independence}
Let $X$ be a projective scheme over a field $k$ and $L$ an $\R$-Cartier 
divisor with a given expression $L\sim_\R \sum t_iA_i$ as in \ref{d-abl-2}. 
Then ${\bf{B}}_+(L)$ as defined in \ref{d-abl-2} satisfies 
the following assertions: 

$(1)$ ${\bf{B}}_+(L)$ does not depend on the choice 
of $A$ nor the expression $L\sim_\R \sum t_iA_i$; 

$(2)$ for any positive rational number $s$ we have ${\bf{B}}_+(sL)={\bf{B}}_+(L)$;

$(3)$ ${\bf{B}}_+(L)=\bigcap {\bf{B}}(L-H)$ where $H$ runs over all ample $\R$-Cartier divisors 
so that $L-H$ is $\Q$-Cartier;

$(4)$ if $L$ is $\Q$-Cartier then ${\bf{B}}_+(L)$ coincides with the one defined in \ref{d-abl}. 
\end{lem}
\begin{proof}
(1) First we show that ${\bf{B}}_+(L)$ is independent of the choice of $A$. Indeed let $G$ be any other ample 
Cartier divisor. Assume $x\notin \bigcap_{m\in \N} {\bf{B}}(\nrd{mL}-A)$. 
Then 
$$
x\notin  {\bf{B}}(\nrd{mL}-A)={\bf{B}}(\sum\rddown{mt_i}A_i-A)
$$ 
for some $m>0$. 
Thus $x\notin  {\bf{B}}(\sum l\rddown{mt_i}A_i-lA)$ for any sufficiently large $l>0$.
Since
$$
\sum \rddown{lmt_i}A_i-\sum l\rddown{mt_i}A_i
$$ 
is zero or ample, $x\notin  {\bf{B}}(\sum \rddown{lmt_i}A_i-lA)$,  
and since $lA$ is sufficiently ample, 
$x\notin  {\bf{B}}(\sum\rddown{lmt_i}A_i-G)={\bf{B}}(\nrd{lmL}-G)$.
This shows that 
$$
\bigcap_{m\in \N} {\bf{B}}(\nrd{mL}-A)\supseteq \bigcap_{m\in \N} {\bf{B}}(\nrd{mL}-G)
$$
The opposite inclusion $\subseteq$ can be proved similarly hence ${\bf{B}}_+(L)$ is independent of $A$. 

Now we show that ${\bf{B}}_+(L)$ is independent of the expression $L\sim_\R\sum t_iA_i$. 
Indeed assume that $L\sim_\R\sum t_i'A_i'$ is another expression. Redefining the indexes 
we can assume that $A_i'=A_i$. Let $A=\sum A_i+G$ with $G$ ample.
Assume that $x\notin \bigcap_{m\in \N} {\bf{B}}(\sum\rddown{mt_i}A_i-A)$. 
Then $x\notin  {\bf{B}}(\sum\rddown{mt_i}A_i-A)$ 
for some $m$ hence $x\notin   {\bf{B}}(\sum l\rddown{mt_i}A_i-lA)$ for any sufficiently large $l>0$.
Arguing as above we can show that $x\notin   {\bf{B}}(\sum \rddown{lmt_i}A_i-A)$. 
Writing 
$lmt_i =\rddown{lmt_i}+u_i$
and $lmt_i' =\rddown{lmt_i'}+u_i'$, we see that  
$$
( \sum\rddown{lmt_i'}A_i-G)-(\sum \rddown{lmt_i}A_i-A)
$$
$$
\sim_\R\sum u_iA_i-\sum u_i'A_i+A-G
$$
is ample hence  $x\notin{\bf{B}}(\sum \rddown{lmt_i'}A_i-G)$ 
so $x\notin \bigcap_{m\in \N} {\bf{B}}(\sum\rddown{mt_i'}A_i-G)$. 
In other words, 
$$
\bigcap_{m\in \N} {\bf{B}}(\sum\rddown{mt_i}A_i-A)\supseteq \bigcap_{m\in \N} {\bf{B}}(\sum\rddown{mt_i'}A_i-G)
$$
The opposite inclusion $\subseteq$ can be proved similarly bearing in mind that we are free 
to change $A$ and $G$. 

(2) 
It is enough to treat the case when $s\in \N$ is sufficiently large. 
It is obvious that ${\bf{B}}_+(sL)\supseteq {\bf{B}}_+(L)$.
Assume that $x\notin {\bf{B}}_+(L)$. Then $x\notin {\bf{B}}(\nrd{mL}-A)$ for some $m$ hence 
$x\notin {\bf{B}}(s\nrd{mL}-sA)$. Arguing as in (1) we see that $x\notin {\bf{B}}(\nrd{msL}-A)$ which 
implies that $x\notin {\bf{B}}_+(sL)$. That is, ${\bf{B}}_+(sL)\subseteq {\bf{B}}_+(L)$.

(3) 
For each $m>0$,  
$$
{\bf{B}}(\nrd{mL}-A)={\bf{B}}(mL-mH_m)={\bf{B}}(L-H_m)
$$ for some ample $\R$-Cartier 
divisor $H_m$. Thus ${\bf{B}}_+(L)\supseteq \bigcap {\bf{B}}(L-H)$.
Conversely assume $x\notin \bigcap {\bf{B}}(L-H)$. Then $x\notin {\bf{B}}(L-H)$ for some $H$. 
Since $L-H$ is assumed to be $\Q$-Cartier, $mL-mH$ is Cartier for some sufficiently divisible $m>0$. 
Since $mH$ is sufficiently ample, ${\bf{B}}(mL-mH)\supseteq {\bf{B}}(\nrd{mL}-A)$ 
hence $x\notin {\bf{B}}(\nrd{mL}-A)$ which implies that $x\notin {\bf{B}}_+(L)$.

(4)
We can write $L\sim_\Q \sum {t_i}A_i$ with all the $t_i$ rational 
numbers. Pick $s\in \N$ so that $sL$ is Cartier, $st_i$ are all integers, and $sL\sim \sum s{t_i}A_i$. 
Then by (2) and (1)  we have
$$
{\bf{B}}_+(L)={\bf{B}}_+(sL)=\bigcap_{m\in \N} {\bf{B}}(\sum\rddown{mst_i}A_i-sA)=\bigcap_{m\in \N} {\bf{B}}(msL-sA)
=\bigcap_{m\in \N} {\bf{B}}(mL-A)
$$ 
But this is the same as ${\bf{B}}_+(L)$ in Definition \ref{d-abl}.\\ 
\end{proof}

\section{Growth of cohomology}

The next lemma is similar to [\ref{CMM}, Lemma 2.2].

\begin{lem}\label{l-h^0}
Let $X$ be a scheme of dimension $d$ projective over a field $k$. 
Assume that $L$ is an $\R$-Cartier divisor with a given expression $L\sim_\R \sum t_iA_i$ as in \ref{d-abl-2}.
Let $\mathcal{F}$ be a coherent sheaf on $X$.  Then the upper growth of $h^0(\mathcal{F}\nrd{mL})$ is 
at most like $m^{d}$.
\end{lem}
\begin{proof} 
By \ref{ss-b-change}, we can extend $k$ hence assume it is infinite.
Let $t$ be a positive integer such that $t_i\le t$ for every $i$. 
By \ref{ss-restriction}, for each $m>0$, we can change the $A_i$ up to linear equivalence so that  
$$
\mathcal{F}\nrd{mL}=\mathcal{F}(\sum \rddown{mt_i}A_i)\subseteq \mathcal{F}(mt\sum A_i)
$$
So by replacing $L$ with $t\sum A_i$ it is enough to assume that $L$ is an effective very ample Cartier 
divisor and enough to show that the upper growth of $h^0(\mathcal{F}(mL))$ is at most like $m^d$. 
Once again by \ref{ss-restriction}, we can change $L$ up to linear equivalence so that the sequence 
$$
0\to \mathcal{F}((m-1)L)\to \mathcal{F}(mL)\to \mathcal{F}(mL)\otimes \mathcal{O}_L\to 0
$$
is exact. Now if $m\gg 0$, by induction on dimension, the upper growth of 
$$
h^0(\mathcal{F}(mL)\otimes \mathcal{O}_L)=h^0(\mathcal{F}(mL))-h^0(\mathcal{F}((m-1)L))
$$ 
is at most like $m^{d-1}$ hence the upper growth of $h^0(\mathcal{F}(mL))$ is at most like $m^d$.\\
\end{proof}

\begin{lem}\label{l-integral-2}
Let $X$ be an integral scheme of dimension $d$ projective over a field $k$. 
Assume that $L$ is an $\R$-Cartier divisor with a given expression $L\sim_\R \sum t_iA_i$ as in \ref{d-abl-2}. 
Then the following are equivalent:

$(1)$ the upper growth of $h^0(\mathcal{O}_X\nrd{mL})$ is like $m^{d}$;

$(2)$ for some coherent sheaf $\mathcal{F}$, 
the upper growth of $h^0(\mathcal{F}\nrd{mL})$ is like $m^{d}$; 

$(3)$ for any coherent sheaf $\mathcal{F}$ whose support is equal to $X$, 
the upper growth of $h^0(\mathcal{F}\nrd{mL})$ is like $m^{d}$; 

$(4)$ $L$ is big.
\end{lem}
\begin{proof}
(1) $\implies$ (3): 
We can extend $k$ to an infinite field hence we can use \ref{ss-restriction}.
Pick an effective sufficiently ample divisor $A$ so that 
$$ 
0\to \mathcal{F}\to \mathcal{F}(A)\to \mathcal{F}(A)\otimes \mathcal{O}_A\to 0
$$  
is exact. Applying Lemma \ref{l-h^0} it is enough to show that the upper growth of 
 $h^0(\mathcal{F}(A)\nrd{mL})$ is like $m^{d}$. Thus by replacing $\mathcal{F}$ 
 with $\mathcal{F}(A)$ we can assume that $\mathcal{F}$ is generated by global sections.
Each global section corresponds to a morphism $\mathcal{O}_X\to \mathcal{F}$. 
Since $X$ is integral, the morphism is injective if and only if its image is not torsion. 
Therefore if $\alpha_1,\dots, \alpha_r$ form a basis of 
$H^0(\mathcal{F})$ and if $\phi_i\colon \mathcal{O}_X\to \mathcal{F}$ corresponds to $\alpha_i$, 
then $\phi_i$ is injective for at least one $i$ otherwise $\mathcal{F}$ would be torsion  
which is not possible as the support of $\mathcal{F}$ is equal to $X$. 
Therefore $h^0(\mathcal{O}_X\nrd{mL})\le h^0(\mathcal{F}\nrd{mL})$ 
which implies that  the upper growth of $h^0(\mathcal{F}\nrd{mL})$ is like $m^d$.

(3) $\implies$ (4): Take $\mathcal{F}=\mathcal{O}_X(-A)$ for some sufficiently ample divisor $A$. 
Then  $h^0(\mathcal{F}\nrd{mL})\gg 0$ if $m\gg 0$ hence in particular 
$h^0(\mathcal{O}_X(\sum \rddown{mt_i}A_i-A))\neq 0$ for some $m\gg 0$. 
Thus since $X$ is integral, $\sum \rddown{mt_i}A_i-A\sim D$ for some effective Cartier divisor $D$. 
Therefore 
$$
L\sim_\R \sum (mt_i-\rddown{mt_i})A_i+A+D
$$
which means that $L$ is big.  

(4) $\implies$ (1): By definition, $L\sim_\R A+D$ where $A$ is an ample $\R$-Cartier divisor and 
$D$ is an effective $\R$-Cartier divisor. 
 Replacing $L$ with a large positive multiple and then replacing $L$ with $\nrd{L}$ allows us to 
assume that $L$ is Cartier and $\nrd{mL}=mL$ for each $m>0$. 
By replacing $A$ we can assume that $L\sim_\Q A+D$ and that $A$ is $\Q$-Cartier. 
But then $D$ is also $\Q$-Cartier and 
$$
h^0(\mathcal{O}_X\nrd{mL})=h^0(\mathcal{O}_X(mL))\ge h^0(\mathcal{O}_X(mA))
$$ 
for any sufficiently divisible $m>0$. Arguing as in the proof of Lemma \ref{l-h^0} we can prove that 
the upper growth of $h^0(\mathcal{O}_X(mA))$ is like $m^d$.

(3) $\implies$ (2): Obvious. 

(2) $\implies$ (1):  There is a filtration 
$$
0=\mathcal{F}_0\subset \mathcal{F}_1\subset \cdots \subset \mathcal{F}_n=\mathcal{F}
$$
of coherent sheaves such that for each $0<j\le n$, there exist a closed embedding $f\colon S\to X$ 
of an integral scheme $S$ and an ideal sheaf $\mathcal{J}\subset \mathcal{O}_{S}$ 
such that $\mathcal{F}_j/\mathcal{F}_{j-1}\simeq f_*\mathcal{J}$ (cf. The stacks project [\ref{stacks}],  
section on d\'evissage of coherent sheaves). Let $j$ be the smallest number such that the upper growth 
of $h^0(\mathcal{F}_j\nrd{mL})$ is like $m^d$. 
Let $f\colon S\to X$ and $\mathcal{J}$ be the corresponding embedding and ideal sheaf 
so that $\mathcal{F}_j/\mathcal{F}_{j-1}\simeq f_*\mathcal{J}$.
Then from the exact sequence 
$$
0\to H^0(\mathcal{F}_{j-1}\nrd{mL})\to H^0(\mathcal{F}_j\nrd{mL})\to H^0(\mathcal{J}\nrd{mL})
$$
we deduce that the upper growth of $h^0(\mathcal{J}\nrd{mL})$ is like $m^d$. By Lemma \ref{l-h^0}, 
$\dim S=d$ hence $S=X$. But then the upper growth of $h^0(\mathcal{O}_X\nrd{mL})$ is like $m^d$.\\
 \end{proof}

\begin{prop}\label{l-growth}
Let $X$ be a projective scheme over a field $k$. 
Assume that $L$ is a nef $\R$-Cartier divisor with a given expression $L\sim_\R \sum t_iA_i$ as in \ref{d-abl-2}. 
Let $\mathcal{F}$ be a coherent sheaf on $X$, and let $Y$ be its support and $e$ the dimension of $Y$. 
Then 

$(1)$ the upper growth of $h^0(\mathcal{F}\nrd{mL})$ is at most like $m^{e}$; 

$(2)$ the upper growth of $h^i(\mathcal{F}\nrd{mL})$ is at most like $m^{e-1}$ for any $i>0$; 

$(3)$ the upper growth of $h^0(\mathcal{F}\nrd{mL})$ is like $m^e$ if and only if 
$L|_Z$ is big for some component $Z$ of $Y$ with $\dim Z=e$.
\end{prop}
\begin{proof}
We can assume that the theorem holds for coherent sheaves with support of 
dimension $<e$. We also may assume that the theorem holds for any closed 
subscheme of $X$ other than $X$ itself.

(1) 
This follows from Lemma \ref{l-h^0}.

(2) 
By extending $k$ we can assume that it is infinite.
Choose an effective sufficiently ample divisor $A$. For each $m>0$ we can write 
$\nrd{mL}+A=A'+A''$ where $A'$ is ample Cartier and $A''$ is sufficiently ample Cartier. 
Then by Fujita vanishing (\ref{ss-Fujita-v}) 
we get $h^i(\mathcal{F}(\nrd{mL}+A))=0$ 
for every $m>0$ and $i>0$.
By \ref{ss-restriction}, we can choose $A$ so that the sequence 
$$
0\to \mathcal{F}\to \mathcal{F}(A)\to \mathcal{F}\otimes\mathcal{O}_A(A)\to 0
$$
is exact. The dimension of the support of $\mathcal{F}\otimes\mathcal{O}_A(A)$ 
is $e-1$. Now using the exact sequence 
$$
 H^{i-1}(\mathcal{F}\otimes\mathcal{O}_A(\nrd{mL}+A)) \to H^i(\mathcal{F}(\nrd{mL}))\to H^i(\mathcal{F}(\nrd{mL}+A))=0
$$
and induction on $e$ we get the result.

(3) 
As in the proof of Lemma \ref{l-integral-2}, there is a filtration 
$$
0=\mathcal{F}_0\subset \mathcal{F}_1\subset \cdots \subset \mathcal{F}_n=\mathcal{F}
$$
of coherent sheaves such that for each $0<j\le n$, there exist a closed embedding $f\colon S\to X$ 
of an integral scheme $S$ and an ideal sheaf $\mathcal{J}\subset \mathcal{O}_{S}$ 
such that $\mathcal{F}_j/\mathcal{F}_{j-1}\simeq f_*\mathcal{J}$.

 Assume that the upper growth of $h^0(\mathcal{F}\nrd{mL})$ is like $m^{e}$. 
Let $j$ be minimal with the property that the upper growth of $h^0(\mathcal{F}_{j}\nrd{mL})$ is like $m^{e}$. 
Let $f\colon S\to X$ and $\mathcal{J}$ be the corresponding embedding and ideal sheaf 
so that $\mathcal{F}_j/\mathcal{F}_{j-1}\simeq f_*\mathcal{J}$.
Then the upper growth of $h^0(\mathcal{J}\nrd{mL})$ is like $m^e$. So in particular 
$\mathcal{J}\neq 0$ and since $S$ is integral the support of $\mathcal{J}$ is equal to $S$. 
Moreover, since the upper growth of $h^0(\mathcal{J}\nrd{mL})$ is like $m^e$, Lemma 
\ref{l-h^0} shows that $\dim S\ge e$. On the other hand, $S$ is a subset of 
$Y$ because $\mathcal{F}|_{X\setminus Y}=0$ and because of the surjection $\mathcal{F}_j\to f_*\mathcal{J}$.  
Thus $\dim S\le e$ hence $\dim S=e$.  Now, by Lemma \ref{l-integral-2}, 
$L|_{S}$ is big and so we can take $Z=S$.

Conversely, assume that there is a component $Z$ of $Y$ of dimension $e$ such that 
$L|_Z$ is big. In the filtration above, let $j$ be the smallest number such that $Z$ 
is a component of the support of $\mathcal{F}_j$. Then $Z$ is a subset of the 
support of the corresponding $f_*\mathcal{J}$ hence $Z\subseteq S$ which in turn 
implies that $Z=S$ because $e=\dim Z\le \dim S\le e$. It is then enough to show that the upper growth of
$h^0(\mathcal{J}\nrd{mL})$ is like $m^e$ because of the exact sequence 
$$
0\to H^0(\mathcal{F}_{j-1}\nrd{mL})\to H^0(\mathcal{F}_j\nrd{mL})\to 
H^0(\mathcal{J}\nrd{mL}) \to H^1(\mathcal{F}_{j-1}\nrd{mL})
$$
and the fact that the upper growth of $h^1(\mathcal{F}_{j-1}\nrd{mL})$ is at most like $m^{e-1}$ 
by (2). Now apply Lemma \ref{l-integral-2}.\\
\end{proof}

\section{Proof of the theorems}

\begin{proof}(of Theorem \ref{t-main-1}) 
 By Noetherian induction we can assume that 
the theorem already holds for any closed subscheme of $X$ not equal to $X$.\\

\emph{Step 1.}
We deal with the first equality in the theorem. By definition, ${\bf{B_+}}(L)\subseteq {\bf{B}}(\nrd{nL}-A)$
 for any $n>0$. Moreover, there are positive integers $m_1,\dots,m_r$ such that 
$$
{\bf{B_+}}(L)={\bf{B}}(\nrd{m_1L}-A)\cap \cdots \cap {\bf{B}}(\nrd{m_rL}-A)
$$
If $n=lm_i$ for some positive integer $l$, then $\nrd{nL}-l\nrd{m_iL}$ is zero or ample hence 
$$
{\bf{B}}(\nrd{nL}-A)\subseteq {\bf{B}}(l\nrd{m_iL}-A)\subseteq {\bf{B}}(l\nrd{m_iL}-lA)= {\bf{B}}(\nrd{m_iL}-A)
$$  
Therefore ${\bf{B_+}}(L)={\bf{B}}(\nrd{nL}-A)$ if each $m_i|n$. 

For the second equality: for any fixed $n'>0$ divisible by all the $m_i$ and any sufficiently divisible $l>0$ we have  
$$
{\bf{B_+}}(L)={\bf{B}}(\nrd{n'L}-A)={{\Bs}}|l\nrd{n'L}-lA|\supseteq {{\Bs}}|\nrd{ln'L}-A|
\supseteq {\bf{B}}(\nrd{ln'L}-A)={\bf{B_+}}(L)
$$
Now take $n=ln'$.\\

\emph{Step 2.}
 The rest of the proof will be devoted to showing 
${\bf{B_+}}(L)=\mathbb{E}(L)$. It is obvious that ${\bf{B_+}}(L)\supseteq\mathbb{E}(L)$
so we will focus on the reverse inclusion.
If $L|_Z$ is not big for every component $Z$ of $X$ (with the reduced induced 
structure), then ${\bf{B_+}}(L)\subseteq\mathbb{E}(L)=X$.
So we may assume that there is a component $Z$ such that $L|_Z$ is big. 
Pick such a $Z$ with maximal dimension, say $e$. Let $Y$ be the union of the other components, again 
with the induced reduced structure.  

There are ideal sheaves $\mathcal{I},\mathcal{J}\subset \mathcal{O}_X$ such that the support of 
$\mathcal{I}$ is inside $Z$ but the support of $\mathcal{O}_X/\mathcal{I}$ is inside $Y$, 
and the support of $\mathcal{J}$ is inside $Y$ but the support of $\mathcal{O}_X/\mathcal{J}$ is inside $Z$ 
(cf. [\ref{stacks}], section on d\'evissage of coherent sheaves). 
Let $Y',Z'$ be the closed subschemes defined by $\mathcal{I},\mathcal{J}$ respectively. 
On $Z\setminus Y$ we have $\mathcal{J}=0$ and $\mathcal{O}_{Z'}=\mathcal{O}_X$. Thus 
the reduced scheme associated to $Z'$ is nothing but $Z$. Similarly, one shows that 
the reduced scheme associated to $Y'$ is $Y$.  By construction, on $Z\setminus Y$ we have 
 $\mathcal{O}_{Z'}=\mathcal{O}_X$ and $\mathcal{I}=\mathcal{O}_X$, and on $Y\setminus Z$ we 
 have $\mathcal{I}=0$.\\

\emph{Step 3.}
We would like to find sections of $\mathcal{O}_X(\nrd{nL}-A)$ which vanish on $Y'$ but not on $Z'$.
Let $\mathcal{I}\to \mathcal{O}_{Z'}$ be the composition 
$\mathcal{I}\hookrightarrow \mathcal{O}_{X} \to \mathcal{O}_{Z'}$
and let $\mathcal{K},\mathcal{L}$ be its kernel and image respectively. 
 Similarly, let $\mathcal{L}\to \mathcal{O}_{Z}$ be the composition
$\mathcal{L}\hookrightarrow \mathcal{O}_{Z'}\to \mathcal{O}_Z$  
and  let $\mathcal{N},\mathcal{M}$ be its kernel and image respectively.  
Then by Step 2 on $Z\setminus Y$ we have  $\mathcal{L}=\mathcal{O}_{Z'}$ and $\mathcal{M}=\mathcal{O}_Z$, 
and on $Y\setminus Z$ we have $\mathcal{L}=\mathcal{M}=0$. 
Therefore the support of 
$\mathcal{L},\mathcal{M}, \mathcal{I}$ are all equal to $Z$, and the support of $\mathcal{K},\mathcal{N}$ 
are subsets of $Z$. 

Now we have the exact sequences 
$$
0\to \mathcal{K}(\nrd{nL}-A)\to \mathcal{I}(\nrd{nL}-A)\to \mathcal{L}(\nrd{nL}-A)\to 0
$$
and 
$$
0\to \mathcal{N}(\nrd{nL}-A)\to \mathcal{L}(\nrd{nL}-A)\to \mathcal{M}(\nrd{nL}-A)\to 0
$$
By Proposition \ref{l-growth}, the upper growth of $h^0(\mathcal{M}(\nrd{nL}-A))$ is like $n^e$ but 
the upper growth of $h^1(\mathcal{N}(\nrd{nL}-A))$ is at most like $n^{e-1}$. 
On the other hand, again by Proposition \ref{l-growth}, the upper growth of 
$h^0(\mathcal{L}(\nrd{nL}-A))$ is like $n^e$ but the upper growth of
$h^1(\mathcal{K}(\nrd{nL}-A))$ is at most like $n^{e-1}$. Therefore 
for infinitely many $n>0$ we can lift a nonzero section of $\mathcal{M}(\nrd{nL}-A)$ to a section of 
$\mathcal{L}(\nrd{nL}-A)$ and in turn to a section of $\mathcal{I}(\nrd{nL}-A)$. 
In other words, there is a section $\alpha\in H^0(\mathcal{I}(\nrd{nL}-A))$ whose restriction to 
$Z$ is nonzero. Since $\mathcal{I}$ is the ideal sheaf of $Y'$, 
$\alpha$ vanishes on $Y'$ when considered as a section of $\mathcal{O}_X(\nrd{nL}-A)$ 
via the injection $\mathcal{I}(\nrd{nL}-A)\to \mathcal{O}_X(\nrd{nL}-A)$.\\

\emph{Step 4.}
From now on we consider $\alpha$ as a section of $\mathcal{O}_X(\nrd{nL}-A)$. 
We can think of $\alpha$ as a morphism $\mathcal{O}_X\to \mathcal{O}_X(\nrd{nL}-A)$ 
such that if we tensor this with $\mathcal{O}_Z$ then we obtain a nonzero morphism. 
Let $\alpha_1:=\alpha$ and let $\mathcal{T}_1$ be the kernel of $\alpha_1$. 
Let $\alpha_2$ be the composition 
$$
\mathcal{O}_X\to \mathcal{O}_X(\nrd{nL}-A)\to \mathcal{O}_X(2\nrd{nL}-2A)\to \mathcal{O}_X(\nrd{2nL}-2A)
$$
where the first morphism is $\alpha_1$, the second one is obtained by 
tensoring $\alpha_1$ with $\mathcal{O}_X(\nrd{nL}-A)$, and the third one 
comes from the choice of an injective morphism 
$\mathcal{O}_X\to \mathcal{O}_X(\nrd{2nL}-2\nrd{nL})$ (which exists because $\nrd{2nL}-2\nrd{nL}$ 
is zero or very ample) and tensoring it with $\mathcal{O}_X(2\nrd{nL}-2A)$. 

Let $\mathcal{T}_2$ be the 
kernel of $\alpha_2$. Obviously, $\mathcal{T}_1\subseteq \mathcal{T}_2$. 
Inductively we can define $\alpha_i$ to be the composition 
$$
\mathcal{O}_X\to \mathcal{O}_X(\nrd{({i-1})nL}-({i-1})A)\to \mathcal{O}_X(\nrd{({i-1})nL}+\nrd{nL}-{i}A) 
\to \mathcal{O}_X(\nrd{{i}nL}-{i}A)
$$ 
where the first map is $\alpha_{i-1}$, the second map is obtained by 
tensoring $\alpha_1$ with $\mathcal{O}_X(\nrd{({i-1})nL}-({i-1})A)$, and third one is obtained from the 
choice of an injective morphism $\mathcal{O}_X\to \mathcal{O}_X(\nrd{{i}nL}-\nrd{({i-1})nL}-\nrd{nL})$. 
Again it is obvious that 
$\mathcal{T}_{i-1}\subseteq \mathcal{T}_i$.\\  

\emph{Step 5.}
By the Noetherian property, there is $r$ such that $\mathcal{T}_r=\mathcal{T}_{r+1}=\cdots$.
Since $\alpha_1$ restricted to $Z$ is nonzero and since $Z$ is integral, we can make sure that the restriction of 
each $\alpha_i$ to $Z$ is also nonzero. Indeed if $U\subset Z$ is a small nonempty open set, 
then the restriction to $U$ of each map in the definition of $\alpha_i$ is an isomorphism.
Therefore each $\alpha_i$ is nonzero  hence $\mathcal{T}_r\subsetneq \mathcal{O}_X$. 

Now tensor $\alpha_r$ with $\mathcal{O}_X(-\nrd{rnL}+rA)$ and let $\mathcal{E}$ be its image 
in $\mathcal{O}_X$. Then we get the exact sequence 
$$
0\to \mathcal{T}_r(-\nrd{rnL}+rA)\to \mathcal{O}_X(-\nrd{rnL}+rA)\to \mathcal{E}\to 0
$$
Let $E$ be the closed subscheme defined by $\mathcal{E}$, that is, $E$ 
is the zero subscheme of $\alpha_r$. Note that $E\neq X$ otherwise $\mathcal{E}=0$ 
hence $\alpha_r=0$ which is not possible.

We will argue that ${\bf{B}}(\nrd{mL}-A)\subseteq E$ if $m>0$ is sufficiently 
divisible. If $E=X$ topologically, then the claim is trivial. So we may assume that $E\neq X$ 
topologically. By construction, $\alpha_r$ does not vanish outside $E$ hence 
${\bf{B}}(\nrd{rnL}-rA)\subseteq E$. If $m=lrn$, then 
$$
{\bf{B}}(\nrd{mL}-A)\subseteq {\bf{B}}(l\nrd{rnL}-A)\subseteq {\bf{B}}(l\nrd{rnL}-lA)
={\bf{B}}(\nrd{rnL}-A)\subseteq {\bf{B}}(\nrd{rnL}-rA)\subseteq E
$$\

\emph{Step 6.}
We will assume that $r\gg 0$. Consider the exact sequence 
$$
0\to \mathcal{T}_r(\nrd{mL}-\nrd{rnL}+rA-aA)\to \mathcal{O}_X(\nrd{mL}-\nrd{rnL}+rA-aA)\to \mathcal{E}(\nrd{mL}-aA)\to 0
$$
Since $\mathcal{T}_r$ does not depend on $r\gg 0$, by Fujita vanishing, we may assume that 
$$
H^i(\mathcal{T}_r(\nrd{mL}-\nrd{rnL}+rA-aA))=0
$$
and 
$$
H^i(\mathcal{O}_X(\nrd{mL}-\nrd{rnL}+rA-aA))=0
$$
for any $i>0$, $m>rn$, and $0\le a\le 1$. Therefore $H^i(\mathcal{E}(\nrd{mL}-aA))=0$ 
if  $i>0$, $m\gg 0$, and $0\le a\le 1$ (in this proof we only need to consider $a=1$ 
but in the proof of Theorem \ref{t-main-2} we need to take $a=0$).

On the other hand, we have the exact sequence 
$$
0\to \mathcal{E}(\nrd{mL}-aA)\to \mathcal{O}_X(\nrd{mL}-aA)\to \mathcal{O}_E(\nrd{mL}-aA)\to 0
$$
from which we obtain the exact sequence 
$$
H^0(\mathcal{O}_X(\nrd{mL}-aA))\to H^0(\mathcal{O}_E(\nrd{mL}-aA))\to H^1(\mathcal{E}(\nrd{mL}-aA))=0
$$
if $m\gg 0$ and $0\le a\le 1$.\\ 

\emph{Step 7.}
From the expression $L\sim_\R \sum t_iA_i$ we obtain the expression $L|_E\sim_\R \sum t_iA_i|_E$. 
For each $m>0$ we get $\nrd{mL}|_E=\nrd{mL|_E}$. 
Taking $a=1$ in Step 6, recalling that 
${\bf{B}}(\nrd{mL}-A)\subseteq E$ if $m>0$ is sufficiently divisible, and using Step 1, we deduce that 
$$
{\bf{B}_+}(L)={\bf{B}}(\nrd{mL}-A)=\Bs|\nrd{mL}-A|
$$
$$
=\Bs|\nrd{mL|_E}-A|_E|
={\bf{B}}(\nrd{mL|_E}-A|_E)={\bf{B}_+}(L|_E)
$$
for any sufficiently divisible $m>0$.

On the other hand, it is easy to see that $\mathbb{E}(L|_E)\subseteq \mathbb{E}(L)$.
Indeed, let $V$ be a component of $\mathbb{E}(L|_E)$. Then $(L|_E)|_V$ is not big so $L|_V$ is not big 
hence $V\subseteq \mathbb{E}(L)$. Finally using the Noetherian induction and the above results we get 
$$
\mathbb{E}(L)\subseteq {\bf{B}}_+(L)={\bf{B}}_+(L|_E)=\mathbb{E}(L|_E)\subseteq \mathbb{E}(L) 
$$
which in particular implies that ${\bf{B}}_+(L)=\mathbb{E}(L)$.\\
\end{proof}

\begin{proof}(of Theorem \ref{t-main-2})
We may assume that the theorem holds for closed subscheme of $X$ other than $X$ itself. 
Moreover, by replacing $L$ with a multiple we can assume that it is Cartier.
If $\mathbb{E}(L)=X$, the theorem is trivial. So we assume this is not the case.
Let $E$ be the subscheme constructed in Step 5 of the proof of Theorem \ref{t-main-1}. 
We showed that if $m\gg 0$, the map 
$$
H^0(\mathcal{O}_X(mL))\to H^0(\mathcal{O}_E(mL))
$$
is surjective (by taking $a=0$). Moreover, we showed that $\mathbb{E}(L|_E)=\mathbb{E}(L)$. 
Since 
$$
{\bf{B}}(L)\subset {\bf{B}_+}(L)=\mathbb{E}(L)\subseteq E
$$
we have ${\bf{B}}(L)={\bf{B}}(L|_E)$. Thus $L$ is semi-ample if and only if $L|_E$ is 
semi-ample. Since the theorem already holds for 
$E$ by assumptions, there is a closed subscheme $Z$ of $E$ whose reduction is  
$\mathbb{E}(L|_E)$ and such that $L|_E$ is semi-ample if and only if $L|_Z$
is semi-ample. Now $L$ is semi-ample if and only if $L|_Z$ is semi-ample.\\
\end{proof}

\begin{proof}(of Theorem \ref{t-main-3})
Assume that $L^d>0$. Let $A$ be a sufficiently ample effective Cartier divisor.
Let $L\sim_\R \sum t_iA_i$ be an expression as in \ref{d-abl-2}. 
First we show that the upper growth of $h^0(\mathcal{O}_X(\nrd{mL}+A))$ is like $m^d$. 
By Fujita vanishing (\ref{ss-Fujita-v}), $h^i(\mathcal{O}_X(\nrd{mL}+A))=0$ for any $i>0$ and $m>0$. 
Thus 
$$
\mathcal{X}(\mathcal{O}_X(\nrd{mL}+A))=h^0(\mathcal{O}_X(\nrd{mL}+A))
$$
On the other hand, by the asymptotic Riemann-Roch theorem the upper growth of $\mathcal{X}(\mathcal{O}_X(\nrd{mL}+A))$
is like the upper growth of $(\nrd{mL}+A)^d$. For each $m>0$, there is an ample $\R$-Cartier divisor 
$G_m$ such that $\nrd{mL}+A=mL+G_m$ hence $(\nrd{mL}+A)^d>m^dL^d$ which shows that the upper 
growth of $(\nrd{mL}+A)^d$ is like $m^d$. Therefore the upper growth of 
$h^0(\mathcal{O}_X(\nrd{mL}+A))$ is like $m^d$.

Now consider the exact sequence 
$$
0\to \mathcal{O}_X(\nrd{mL}-A)\to \mathcal{O}_X(\nrd{mL}+A)\to \mathcal{O}_{2A}(\nrd{mL}+A)\to 0
$$ 
By Lemma \ref{l-h^0}, the upper growth of $h^0(\mathcal{O}_{2A}(\nrd{mM}+A))$ is at most like 
$m^{d-1}$ hence the upper growth of $h^0(\mathcal{O}_X(\nrd{mL}-A))$ is like $m^d$. 
Since $X$ is integral, $\nrd{mL}-A\sim D$ for some effective Cartier divisor $D$. Therefore 
$mL\sim_\R A'+D$ for some ample $\R$-Cartier divisor $A'$ hence $L$ is big.

Conversely assume that $L$ is big. So by definition $L\sim_\R A'+D$ for some ample $\R$-Cartier divisor $A'$ 
and effective $\R$-Cartier divisor $D$. By replacing $L,A',D$ appropriately we can assume that 
$A'$ is a general very ample Cartier divisor. Let $S$ be an irreducible component of $A'$ with the 
induced reduced structure. Since $\dim S=d-1$ and since $L|_S$ is big, by induction on dimension, 
$(L|_S)^{d-1}>0$ hence $L^{d-1}\cdot S>0$ which implies that $L^d=L^{d-1}\cdot A'+L^{d-1}\cdot D>0$.\\
\end{proof}


\section{The augmented base locus of log divisors}

Assume that $X$ is a normal projective variety of dimension $d$ over an algebraically 
closed field $k$, and that $B,A\ge 0$ are $\R$-divisors. 
Moreover, suppose $A$ is nef and big and $L=K_X+B+A$ is nef. 

\begin{thm}
Assume $L^d=0$. Then  ${\bf{B}}_+(L)=X$ is covered by rational curves $C$ with $L\cdot C=0$. 
\end{thm}

The theorem was proved by  Cascini-Tanaka-Xu [\ref{CTX}] and independently 
by M$^{\rm c}$Kernan, when $k$ has positive characteristic.
A short proof of this in any characteristic was given in [\ref{B-mmodel-char-p}].
Now if $L^d>0$, what can we say about ${\bf{B}}_+(L)$? For example, is it 
again covered by rational curves intersecting $L$ trivially? We give a couple of examples to 
shed some light on this question.

\begin{exa}
Let $E$ be an elliptic curve over an algebraically closed field $k$ and let 
$X=\PP(\mathcal{O}_E\oplus \mathcal{O}_E(1))$. The surjection 
$\mathcal{O}_E\oplus \mathcal{O}_E(1)\to \mathcal{O}_E$ defines a section of the 
projection $X\to E$ whose image will be denoted by $E$ again. 
Moreover, there is a birational contraction $X\to Z$ which contracts only $E$. 
Let $B=E$ and $A$ be the pullback of a sufficiently ample divisor on $Z$. 
Let $L=K_X+B+A$. By construction, ${\bf{B}}_+(L)=\mathbb{E}(L)=E$ which is not covered by rational 
curves but at least it is covered by curves intersecting $L$ trivially. 
\end{exa}

\begin{exa}
There is a projective surface $S$ over an algebraically closed field $k$ so that 
there is a nef prime divisor $M$ such that $\kappa(M)=0$, $K_S+2M\sim 0$,  
and if $M\cdot C=0$ for any curve $C$ then $C=M$ 
(see Shokurov [\ref{Sh-complements}, Example 1.1] for such an example).
Let $X=\PP(\mathcal{O}_S\oplus \mathcal{O}_S(1))$. The surjection 
$\mathcal{O}_S\oplus \mathcal{O}_S(1)\to \mathcal{O}_S$ defines a section of the 
projection $\pi\colon X\to S$ whose image will be denoted by $S$ again. 
Moreover, there is a birational contraction $X\to Z$ which contracts only $S$ to a point. 
Let $B=S+3\pi^*M$ and let $A$ be the pullback of a sufficiently ample divisor on $Z$. 
Let $L=K_X+B+A$. Then ${\bf{B}}_+(L)=\mathbb{E}(L)\subseteq S$. Moreover, since 
$$
L|_S=(K_X+S+3\pi^*M+A)|_S\sim K_S+3M\sim M 
$$
is not big, ${\bf{B}}_+(L)=\mathbb{E}(L)= S$. But there is no family of curves $C$ covering $S$ 
with the property $L\cdot C=0$. 
\end{exa}

These examples show that we need to put some reasonably strong condition on $X,B,A$ 
to be able to say something interesting about ${\bf{B}}_+(L)$.

\begin{quest}
Assume that $(X,B)$ is a projective klt pair over an algebraically closed 
field $k$ and $A$ a nef and big $\R$-divisor.  
Assume that $L=K_X+B+A$ is nef and that $L^d>0$. 
Is it true that  ${\bf{B}}_+(L)$ is covered by rational curves $C$ with $L\cdot C=0$?
\end{quest} 

Assume that $k=\C$. Then $L$ in the question is semi-ample by the base point free theorem 
hence it defines a contraction $X\to Y$. Moreover, it is well-known that the fibres of $X\to Y$ are covered 
by rational curves. 
Note that ${\bf{B}}_+(L)$ is nothing but the union of the fibres of $X\to Y$. 
 
Now assume that $k$ has characteristic $p>5$ and $\dim X\le 3$. 
One can show that $L$ is again semi-ample (if $\dim X=2$, this holds for any $p$ [\ref{Tanaka}]).
We sketch the proof.
Since $A$ is nef and big, we can change the situation so that it is ample [\ref{B-mmodel-char-p}, Lemma 8.2].
Using boundedness of the length of extremal rays [\ref{Keel}][\ref{B-mmodel-char-p}, 3.3]
one can show that $L=\sum r_i(K_X+B_i+A_i)$ where $r_i>0$, $\sum r_i=1$, $B_i,A_i$ are effective  
$\Q$-divisors, $A_i$ is ample, $(X,B_i)$ is klt, and $K_X+B_i+A_i$ is nef and big. 
Now applying [\ref{B-mmodel-char-p}, Theorem 1.4][\ref{Xu}] each $K_X+B_i+A_i$ is semi-ample 
hence $L$ is also semi-ample. Thus $L$ defines a contraction $X\to Y$. In particular, ${\bf{B}}_+(L)$ is covered by a 
family of curves intersecting $L$ trivially. Using the results of [\ref{B-mmodel-char-p}] it does not seem
hard to prove that the fibres of $X\to Y$ are actually covered by rational curves.

Assume that $k$ has positive characteristic and $\dim X\ge 4$. It seems hard to answer the question in this 
case because of the lack of resolution of singularities. However, if replace the klt condition 
with strongly $F$-regular, then it is likely that one can actually answer the question.


\vspace{2cm}

\flushleft{DPMMS}, Centre for Mathematical Sciences,\\
Cambridge University,\\
Wilberforce Road,\\
Cambridge, CB3 0WB,\\
UK\\
email: c.birkar@dpmms.cam.ac.uk\\

\end{document}